\documentclass{amsart}

\usepackage{amssymb}
\usepackage{amsfonts}
\usepackage{amsmath}
\usepackage{amsthm}
\usepackage{graphicx}
\usepackage{mathrsfs}
\usepackage{dsfont}
\usepackage{amscd}
\usepackage{multirow}
\usepackage{palatino}
\usepackage{mathpazo}
\usepackage{mathtools}
\usepackage[all]{xy}
\usepackage[scaled]{helvet} 
\usepackage{courier} 
\normalfont
\usepackage[T1]{fontenc}
\usepackage{calligra}
\usepackage{verbatim}
\usepackage[usenames,dvipsnames]{color}
\usepackage[colorlinks=true,linkcolor=Blue,citecolor=Violet]{hyperref}
\usepackage{cite}
\usepackage{enumitem}

\makeatletter
\@namedef{subjclassname@2010}{%
  \textup{2010} Mathematics Subject Classification}
\makeatother

\theoremstyle{plain}
\newtheorem{theorem}{Theorem}[section]

\newtheorem{corollary}[theorem]{Corollary}

\newtheorem{lemma}[theorem]{Lemma}
\newtheorem{proposition}[theorem]{Proposition}

\newtheorem{theorem-definition}[theorem]{Theorem-Definition}

\theoremstyle{definition}
\newtheorem{definition}[theorem]{Definition}

\newtheorem{convention}[theorem]{Convention}

\theoremstyle{remark}

\numberwithin{equation}{section}
 
\newcommand{\A}{{\mathcal{A}}}

\newcommand{\un}{{\mathds{1}}}

\newcommand{\N}{{\mathds{N}}}

\newcommand{\R}{{\mathds{R}}}
\newcommand{\C}{{\mathds{C}}}

\renewcommand{\geq}{\geqslant}
\renewcommand{\leq}{\leqslant}

\allowdisplaybreaks[1]

\makeatletter
\newcommand{\vast}{\bBigg@{4}}
\newcommand{\Vast}{\bBigg@{5}}
\makeatother

\begin{document}

\title[The Bures metric and the quantum metric]{On the Bures metric, C*-norm, and the quantum metric}
\author{Konrad Aguilar}
\address{Department of Mathematics and Statistics, Pomona College, 610 N. College Ave., Claremont, CA 91711} 
\email{konrad.aguilar@pomona.edu}
\urladdr{\url{https://aguilar.sites.pomona.edu}}
\thanks{The first author is supported by NSF grant DMS-2316892}

\author{Karina Behera}
\address{Department of Mathematics and Statistics, Pomona College, 610 N. College Ave., Claremont, CA 91711} 
\email{kaba2021@mymail.pomona.edu}

\author{Tron Omland}
\address{Norwegian National Security Authority (NSM) \and Department of Mathematics, University of Oslo, Norway}
\email{tron.omland@gmail.com}

\author{Nicole Wu}
\address{Department of Mathematics, Harvey Mudd College, 320 E. Foothill Blvd., Claremont, CA 91711} 
\email{nwu@hmc.edu}

\maketitle

\begin{abstract}
    We prove that the topology on the density space with respect to a unital C*-algebra and a faithful induced by the C*-norm is finer than the Bures metric topology. We also provide an example when this containment is strict. Next, we provide a metric on the density space induced by a quantum metric in the sense of Rieffel and prove that the induced topology is the same as the topology induced by the Bures metric and C*-norm when the C*-algebra is assumed to be finite dimensional. Finally, we provide an example of when the Bures metric and induced quantum metric are not metric equivalent. Thus, we provide a bridge between these aspects of quantum information theory and noncommutative metric geometry.
\end{abstract}

\tableofcontents

\baselineskip=17pt

\section{Introduction and Background} 

The Bures metric, which was introduced by Bures in \cite{Bures}, is a vital tool in quantum information theory (see \cite{Hayashi} for some applications). Recently, the Bures metric has been adapted to world of von Neumann algebras and C*-algebras by Farenick and Rahaman \cite{Farenick-Rahaman17}. The quantum metric, which was introduced by Rieffel in  \cite{Rieffel98a}, allows one to prove powerful results about convergence of quantum spaces including those arising from high energy physics \cite{Rieffel00, Rieffel01} and many more \cite{Latremoliere-Packer17, Latremoliere20, Kaad20, Aguilar24, Latremoliere24b} to only name a few.  This paper brings these two important metrics together by using a natural way to compare these two metric along with some comparisons between their topological and geometric properties in the hopes to introduce these methods of measurement to the other field.

In Section \ref{s:c*-bures}, we provide comparisons of the topological structure of the Bures metric and a third metric, the one induced by the C*-norm on the density space. The purpose of bringing this third metric into the picture is not just to provide another interesting comparison, but to also provide a method to compare the Bures metric and the quantum metric in the next section. In Section \ref{s:top-c*-bures-quantum}, we show how one can place a metric on the density space using quantum metrics, and then show that this induced quantum metric is topologically equivalent to the Bures metric. Our approach uses the results in  Section \ref{s:c*-bures}. Finally, in Section \ref{s:non-metric}, we provide a case when the Bures metric and induced quantum metric are not metric equivalent. For the remainder of this section, we provide some necessary background for the rest of the paper.

\begin{convention}
Given a unital C*-algebra $\A$, we denote its unit by $1_\A$, its norm by $\|\cdot\|_\A$, its self-adjoint elements by $sa(\A)$ and positive elements by $\A_+$, and  its state space by $S(\A)$. We also denote the metric induced by $\|\cdot\|_\A$ by $d_\A.$

However,    given a compact Hausdorff space $X$, we denote the C*-norm on $C(X)$ by $\|\cdot\|_X$ and the unit (the constant $1$ function), by $\un$.
\end{convention}

\begin{definition}[\! \cite{Rieffel98a}]\label{d:quantum-metric}
    Let $\A$ be a unital $C^*$-algebra. If  $L:\A \rightarrow [0,\infty]$ is a seminorm on $\A$ such that:
    \begin{enumerate}
        \item $dom(L)=\{a \in \A : L(a)< \infty\}$ is dense in $\A$,
        \item $L(a)=0$ if and only if $a \in \C1_\A$,
        \item $L(a)=L(a^*)$ for every $a \in \A$,
        \item the Monge-Kantorovich metric defined for any two states $\varphi, \psi \in S(\A)$ by \[mk_L(\varphi, \psi)=\sup \{ |\varphi(a)- \psi(a)| : a \in \A, L(a) \leq 1\}
        \]
        metrizes the weak* topology on $S(\A),$
    \end{enumerate}
then $(\A, L)$ is a {\em compact quantum metric space}, and we call $L$ an $L$-seminorm.
\end{definition}

\begin{theorem-definition}[\! {\cite[Theorem 2.6]{Farenick-Rahaman17}}]\label{d:bures-metric}
    Let $\A$ be a unital $C^*$-algebra. Let $\tau$ be a faithful trace (not necessarily of norm $1$). Define the  {\em density space with respect to $\tau$} to be
    \[
    D_\tau(\A)=\{ a \in \A_+ : \tau(a)=1\}.
    \]
    Define the {\em Bures metric with respect to $\tau$} for every $x,y \in D_\tau(\A)$ by 
    \[
    d_B^\tau (x,y)=\sqrt{1-\tau\left(|\sqrt{x}\sqrt{y}|\right)}.
    \]
\end{theorem-definition}
We call the above a "Theorem-Definition" since the proof that the Bures metric is indeed a metric in this general setting of unital C*-algebras equipped with faithful trace is a non-trivial result of \cite{Farenick-Rahaman17}.

We also formally state what we mean by topological equivalence and metric equivalence so that there is no confusion.
\begin{definition}
    Let $X$ be a non-empty set and let $d$ and $d'$ be two metrics on $X$.
    \begin{enumerate}
        \item We say that $d$ and $d'$ are {\em topologically equivalent} if they induce the same topologies.
        \item We say that $d$ and $d'$ are {\em metric equivalent} if there exist $\alpha,\beta>0$ such that 
        \[
        \alpha d(x,y) \leq d'(x,y) \leq \beta d(x,y)
        \]
        for every $x,y \in X$, or equivalently
        \[
        \alpha \leq \frac{d'(x,y)}{d(x,y)} \leq\beta
        \]
        for every $x,y \in X$ such that $x \neq y.$
    \end{enumerate}
\end{definition}

\section*{Acknowledgements}
This work is partially supported by the first author's  NSF grant DMS-2316892. The third author would like to thank Erik B\'edos for helpful discussions.

\section{Comparison of C*-metric  and Bures metric topologies}\label{s:c*-bures}
In this section, we show that the C*-metric topology is finer than the topology induced by the Bures metric. We also show that this containment can be strict by providing an explicit example in the C*-algebra of complex-valued continuous functions on $[0,1]$ and the trace given by Lebesgue integration. First, we prove two lemmas that are likely well-known, but we provide their proofs for convenience.

\begin{lemma}\label{l:functional-calculus}
    Let $\A$ be a unital $C^*$-algebra. 
Let $(x_n)_{n \in \N}$ be a sequence in $\A_+$ that converges in the $C^*$-norm to   $x \in \A_+$.
Let $r \geq 0$ be such that for any $n \in \N$, $\| x_n \|_\A, \| x \|_\A \leq r$.
If $f:[0,r] \rightarrow \mathbb{R}$ is   continuous, then $(f(x_n))_{n \in \N}$ converges to $f(x)$ in the $C^*$-norm.
\end{lemma}

\begin{proof}
By Weierstrass approximation theorem, there exists polynomials $(p_n)_{n \in \N}$ that converges uniformly to $f$ on $[0, r]$.

Let $\varepsilon > 0$. By uniform convergence of $(p_n)_{n \in \N}$ to $f$, there exists $N \in \N$ such that, for any $n \geq N$, $\|p_n - f\|_{[0,r]} < \frac{\varepsilon}{3}$, and in particular $\|p_N - f\|_{[0,r]} < \frac{\varepsilon}{3}$.

Note that as $p_N$ is a polynomial, there exits $N^{\prime} \geq N$ such that, for any $n \geq N^{\prime}$, \[\|p_N (x_n) - P_N (x)\|_\A < \frac{\varepsilon}{3}.\]

Let $n \geq N'$. Then by functional calculus,
\[
\|f(x_n) - p_N(x_n)\|_\A =\|f - p_N\|_{\sigma(x_n)}\leq \|f-p_N\|_{[0,r]}<\frac{\varepsilon}{3}
\]
since $\sigma(x_n) \subseteq [0,r]$. 
Similarly,
\[
\|f(x) - p_N(x)\|_\A<\frac{\varepsilon}{3}
\]
since $\sigma(x)\subseteq [0,r]$.

By triangle inequality of the norm,
\begin{equation*} \label{eq1}
\begin{split}
\|f(x_n) - f(x)\|_\A & \leq \|f(x_n) - p_N(x_n)\|_\A + \|p_N(x_n) - p(x)\|_\A \\
& \quad \quad + \|p_N(x) - f(x)\|_\A \\
 & < \frac{\varepsilon}{3} + \frac{\varepsilon}{3} + \frac{\varepsilon}{3}   = \varepsilon.
\end{split}
\end{equation*}

Thus, for any arbitrary $\varepsilon > 0$, we can find $N^{\prime}$ such that for any $n \geq N^{\prime}$, $\|f(x_n) - f(x)\|_\A \leq \varepsilon$. Thus, $(f(x_n))_{n \in \N}$ converges to $f(x)$ in the $C^*$-norm.
\end{proof}

\begin{lemma}\label{l:continuous-absolute}
    Let $\A$ be a unital $C^*$-algebra. 
If $(x_n)_{n \in \N}$ be a sequence in $\A$ that converges in the $C^*$-norm to a positive element $x \in \A$, then $(|x_n|)_{n \in \N}$ $C^*$-norm converges to $|x|$.
\end{lemma}

\begin{proof}
Note that $|x| = \sqrt{x^* x}$ and $|x_n| = \sqrt{x_n^* x_n}$ for every $n \in \N$.
For any $n \in \N$, let $a_n = x_n^* x_n$ and $a = x^* x$. 

Then, since multiplication and the adjoint are continuous in the $C^*$-norm, we have that $(a_n)_{n \in \N}$ converges to $a$ in the $C^*$-norm.

For any $n \in \N$, $|x| = \sqrt{a}$ and $|x_n| = \sqrt{a_n}$. Therefore, by Lemma \ref{l:functional-calculus} and   continuity of the square root, we know that $(|x_n|)_{n \in \N}$ $C^*$-norm converges to $|x|$.
\end{proof}

We now prove our main theorem that allows us to compare the C*-metric topology and the Bures metric topology.

\begin{theorem}\label{t:conv-bures-c*}
Let $\A$ be a unital $C^*$-algebra and let $\tau$ be a faithful tracial state.
Let $(x_n)_{n \in \N}$ be a sequence in $D_{\tau}(\A)$ and let $x \in D_{\tau}(\A)$.
If $(x_n)_{n \in \N}$ $C^*$-norm converges to $x$, then $(x_n)_{n \in \N}$ converges to $x$ with respect to the Bures metric $d_{B}^{\tau}$.
\end{theorem}

\begin{proof}
If $(x_n)_{n \in \N}$ converges to $x$ in the $C^*$-norm, 
then, since $g(x) = \sqrt{x}$ is   continuous, 
   $(\sqrt{x_n})_{n \in \N}$ converges to $\sqrt{x}$ in the $C^*$-norm by Lemma \ref{l:functional-calculus}.

Since $x$ is a positive and fixed,  $(\sqrt{x_n}\cdot \sqrt{x})_{n \in \N}$ converges to $\sqrt{x} \cdot \sqrt{x}$ in the $C^*$-norm. Note that $\sqrt{x} \cdot \sqrt{x} = x$   by definition, so $(\sqrt{x_n}\cdot \sqrt{x})_{n \in \N}$  $C^*$-norm converges to $x$. By Lemma \ref{l:continuous-absolute}, this implies $(|\sqrt{x_n}\cdot \sqrt{x}|)_{n \in \N}$ $C^*$-norm converges to $|x|=x$. 

Since $\tau$ is $C^*$-norm continuous,   $(|\sqrt{x_n}\cdot \sqrt{x}|)_{n \in \N}$ $C^*$-norm converging to $x$ implies that $(\tau (|\sqrt{x_n}\cdot \sqrt{x}|))_{n \in \N}$ converges to $\tau(x)$. Since $x \in D_{\tau}(A)$, we have $\tau(x) = 1$. Hence, $(\tau (|\sqrt{x_n}\cdot \sqrt{x}|))_{n \in \N}$ converges to $1$.

Consider \[(d_{B}^{\tau}(x_n,x))_{n \in \N} = \left(\sqrt{1 - \tau\left(|\sqrt{x_n}\cdot \sqrt{x}|\right)}\right)_{n \in \N}.\]   Since \[\left(\tau(|\sqrt{x_n}\cdot \sqrt{x}|)\right)_{n \in \N}\]converges to $1$,  \[(d_{B}^{\tau}(x_n,x))_{n \in \N} \] converges to $\sqrt{1-1} = 0$. Hence $(x_n)_{n \in \N}$ converges to $x$ with respect to the Bures metric $d_{B}^{\tau}$.
\end{proof}

We thus have.

\begin{corollary}\label{c:c*-norm-finer}
Let $\A$ be a unital C*-algebra. Let $\tau$ be a faithful tracial state. 

It holds that the topology induced by $d_\A$ is finer than the topology induced by $d_B^\tau$.
\end{corollary}

 This raises the question of whether these topologies are the same in general. The answer is no, and the rest of this section is devoted to providing an example of when these topologies disagree.

 Consider, $C([0,1])$, the C*-algebra of continuous complex-valued functions  on $[0,1]$.  
The map
\[
\rho: f \in C([0,1]) \longmapsto \int_0^1 f(x) \ dx,
\]
where $\int_0^1 \cdot  \ dx$ is the standard Lebesgue integral,
is a faithful trace on $C([0,1])$.

\begin{proposition}\label{p:strictly-finer}
The topology induced by the C*-norm induced metric, $d_{C([0,1])}$, on the density space $D_\rho(C([0,1]))$ is strictly finer than the topology induced by the Bures metric $d_B^\rho.$
\end{proposition}
\begin{proof}
    The fact that the topology induced by $d_{C([0,1])}$ is finer is provided by Corollary \ref{c:c*-norm-finer}. Thus, it remains to show that the topologies don't equal.

    To accomplish this, we will find a sequence in $D_\rho(C([0,1]))$ that converges with respect to $d_B^\rho$, but does not converge uniformly (i.e. with respect to $d_{C([0,1])}$).

Let $n \in \N$. 
Consider
\begin{equation}\label{eq:sequence}
f_{n}(x)=
    \begin{cases}
        2nx & \text{if } x \in [0, \frac{1}{2n}]\\
        1 & \text{if } x \in (\frac{1}{2n}, 1-\frac{1}{2n})\\
				2nx-2n+2 & \text{if } x \in [1-\frac{1}{2n}, 1]
    \end{cases}
\end{equation}
defined for all $x \in [0,1]$ (see Figure \ref{f:sequence}). Note that $f_n \in C([0,1]).$

\begin{figure}[h]
    \centering
    \includegraphics[width=8cm]{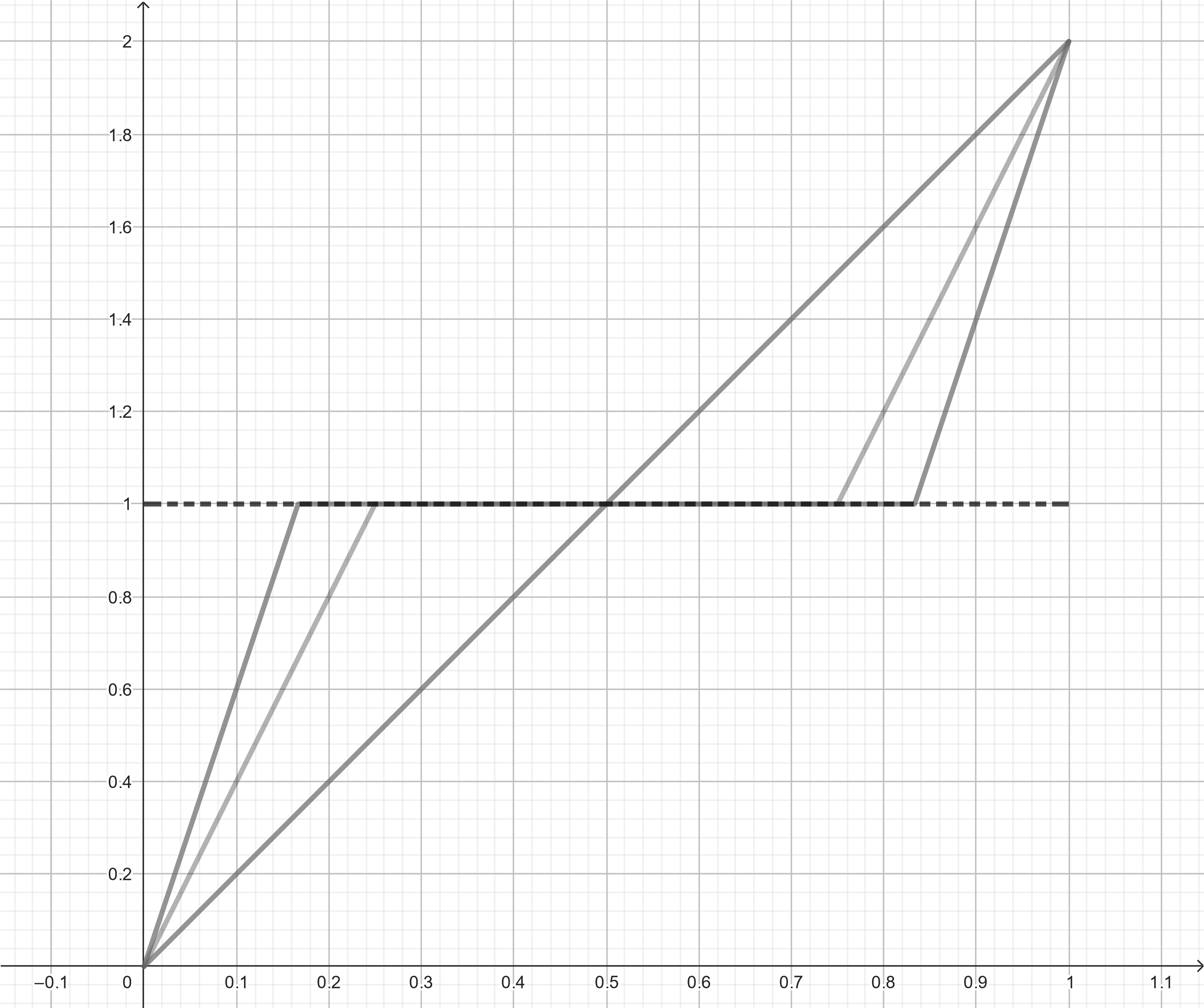}
    \caption{$f_1,f_2,f_3$ of Expression \eqref{eq:sequence}. (Plotted in GeoGebra)}
    \label{f:sequence}
\end{figure}

Define $f:[0,1]\rightarrow \R$ defined by $f(x)=1$ for all $x \in [0,1].$ Note $f \in D_\rho(C([0,1])).$

Next, we check that $f_n\in D_\rho(C([0,1])$ for every $n \in \N$. Let $n\in \N$. First $f_n\geq 0$ by construction. Second, we have \begin{equation*}
\begin{split}
\rho(f_n) &= \int_{0}^{1} f_n(x) \,dx \\
&= \int_{0}^{\frac{1}{2n}} 2nx \,dx + \int_{\frac{1}{2n}}^{1 - \frac{1}{2n}} 1 \,dx + \int_{1 - \frac{1}{2n}}^{1} (2nx-2n+2) \,dx \\
&= [nx^2]_0^{\frac{1}{2n}} + [x]_{\frac{1}{2n}}^{1 - \frac{1}{2n}} + [nx^2-2nx+2x]_{1 - \frac{1}{2n}}^{1} \\
&= \frac{1}{4n} + 1 - \frac{1}{n} + \frac{3}{4n} \\
&= 1
\end{split}
\end{equation*}
Hence $f_n \in D_\rho(C([0,1])).$

We will now prove that the sequence of functions $(f_n)_{n\in \N}$ converges to the function $f $ in the Bures metric, $d_B^\rho$, but fails to converge to  $f$ uniformly.

Let $n \in \N.$ Then
    \begin{equation*}
\begin{split}
\rho(\sqrt{f_n}\sqrt{f}) &= \int_{0}^{1} (\sqrt{f_n}\sqrt{f}) \,dx \\
&= \int_{0}^{1} (\sqrt{f_n}) \,dx \\
&= \int_{0}^{\frac{1}{2n}} (\sqrt{2nx}\cdot 1) \,dx + \int_{\frac{1}{2n}}^{1 - \frac{1}{2n}} (\sqrt{1}\cdot 1) \,dx \\
& \quad \quad + \int_{1 - \frac{1}{2n}}^{1} (\sqrt{2nx-2n+2}\cdot 1) \,dx \\
&= \int_{0}^{\frac{1}{2n}} \sqrt{2nx} \,dx + \int_{\frac{1}{2n}}^{1 - \frac{1}{2n}} 1 \,dx + \int_{1 - \frac{1}{2n}}^{1} \sqrt{2nx-2n+2} \,dx \\
&= \sqrt{2n} \cdot [\frac{2}{3} x^{\frac{3}{2}}]_0^{\frac{1}{2n}} + [x]_{\frac{1}{2n}}^{1 - \frac{1}{2n}} + \sqrt{2} \cdot [\frac{2}{3n} \cdot (nx-n+1)^{\frac{3}{2}}]_{1 - \frac{1}{2n}}^{1} \\
&= \frac{1}{3n} + 1 - \frac{1}{n} + \frac{2\sqrt{2}}{3n} - \frac{1}{3n} \\
&= 1 - \frac{3-2\sqrt{2}}{3n}.
\end{split}
\end{equation*}

Therefore, \[d_{B}^{\rho}(f_n,f) = \sqrt{1 - \rho(\sqrt{f_n}\sqrt{f})} = \sqrt{1 - \left(1 - \frac{3-2\sqrt{2}}{3n}\right)} = \sqrt{\frac{3-2\sqrt{2}}{3n}}.\]

Thus, \[\lim_{n \to \infty}d_{B}^{\rho}(f_n,f) =0.\]
Therefore, we have shown that $(f_n)_{n \in \N}$ converges to $f $ in the $\rho$-Bures metric.

Finally, suppose by way of contradiction that  $(f_n)_{n\in \N}$ converges uniformly to $f $. Then by definition, for any $\delta > 0$, there exists $M \in \N$ such that for any $n > M$ and any $x \in [0,1]$, we have $|f_n(x) - f| < \delta$. Therefore, for some $\delta_0 \in (0,1)$, there exists $M_0 \in \N$ such that for any $n > M_0$ and any $x \in [0,1]$, we have $|f_n(x) - f| < \delta_0$. However, let $x = 1$, then $|f_n(1) - f(1)| = |2 - 1| = 1 > \delta_0$, a contradiction. 
Hence, our assumption is false.
\end{proof}

\section{Topological Equivalence of Bures, C*-, and quantum metrics for finite-dimensional C*-algebras}\label{s:top-c*-bures-quantum}

Although we just saw an example of when the C*-metric topology and Bures metric topology do not agree on the density space, in this section, we will see that if the C*-algebra is assumed to be finite dimensional, then we will show that these topologies agree. Moreover, we will also show that in the finite-dimensional case, we will get that the Bures metric agrees with a metric on the density space induced by the quantum metric, which thus brings these two important metrics together.  We first show how one can induce a metric on the density space using a quantum metric in a natural way.

Let $\A$ be a unital C*-algebra. Let $\tau$ be a faithful trace  on $\A$. For each $a \in D_\tau(\A)$, define
\[
\varphi_a^\tau (b)= \tau(ab)
\]
for every $b \in \A$.

The next result might be well-known, but we provide a proof here for convenience.
\begin{proposition}\label{p:density-state-map}
Let $\A$ be a unital C*-algebra. Let $\tau$ be a faithful trace   on $\A$. 
The map 
\[
\Phi_\tau : a \in D_\tau(\A) \longmapsto \varphi_a^\tau \in S(\A)
\]
is well-defined and injective. 
\end{proposition}
\begin{proof}
   Let $a\in \A_+$ and so there exists $b\in \A$ such that $a=b^*b$ by \cite[Lemma I.4.3]{Davidson}. Let $c\in \A$. Then since $\tau$ is a trace,
   \[
   \varphi_a^\tau(c^*c)= \tau(b^*bc^*c)= \tau(cb^*bc^*)= \tau(cb^*(cb^*)^*)\geq 0.
   \]
   Thus $\varphi$ is a positive linear functional, and in particular, we have that $\varphi_a^\tau(1_\A)=\|\varphi_a^\tau\|_\mathrm{op}$ by \cite[Lemma I.9.5]{Davidson}.  Hence as $a \in D_\tau(\A)$
   \[
   \|\varphi_a^\tau\|_\mathrm{op}=\tau(a1_\A)=\tau(a)=1.
   \]
 Thus $\varphi_a^\tau$ is a state, and therefore, $\Phi_\tau$ is well-defined.

 Next, let $a,a'\in \A_+$ such that $\Phi_\tau(a)=\Phi_\tau(a')$. Hence $\varphi_a^\tau((a-a')^*)=\varphi^\tau_{a'}((a-a')^*)$. Thus $\tau(a(a-a')^*)=\tau(a'(a-a')^*)$, and so $\tau((a-a')(a-a')^*)=0$. Therefore, $a-a'=0$ as $\tau$ is faithful. Hence $a=a'.$  Thus $\Phi_\tau$ is injective.
\end{proof}
This allows us to define a new metric on $D_\tau(\A).$ Indeed

\begin{definition}\label{d:quantum-density-metric}
    Let $(\A, L)$ be a compact quantum metric space. For every $x,y \in D_\tau(\A)$, define
    \[
    d_L^\tau(x,y)=mk_L(\Phi_\tau(x),\Phi_\tau(y))=mk_L(\varphi_x,\varphi_y),
    \]
    which defines a metric on $D_\tau(\A)$ since $\Phi_\tau$ is well-defined and injective. We will still call $d_L^\tau$ a quantum metric.
\end{definition}

For the remainder of the section $\A$ will be a finite-dimensional C*-algebra.
We will now prove that the C*-metric topology and quantum metric topology are the same in this case. But first, we establish that $(D_\tau(\A), d_\A)$ is a compact metric space.

\begin{proposition}\label{p:c*-metric-compact-fd}
   If $\A$ is a finite-dimensional C*-algebra, then  $(D_\tau(\A), d_\A)$ is a compact metric space.
\end{proposition}
\begin{proof}
    We only need to show that $D_\tau(\A)$ is closed and bounded with respect to the C*-norm. First, it is closed since $\tau$ is continuous with respect to the C*-norm.

Since $\A$ is finite dimensional, we have that the norms $\|\cdot\|_\A$ and $\|\cdot\|_\tau$ are equivalent, where $\|a\|_\tau=\sqrt{\tau(a^*a)}$ for every $ a\in \A$.  Now, let $a \in D_\tau(\A)$ and let $\sqrt{a}$ denote its unique positive square root.  We have by the C*-identity
    \begin{align*}
    \sqrt{\|a\|_\A}& =  \sqrt{\|(\sqrt{a})^*\sqrt{a}\|_\A}\\
    &=\|\sqrt{a}\|_\A \\
    &\leq \alpha \|\sqrt{a}\|_\tau\\
    &=\alpha\sqrt{\tau((\sqrt{a})^*\sqrt{a})}=\alpha\sqrt{\tau(a)}=\alpha
    \end{align*}
    since $\tau(a)=1$. 
    And so $\|a\|_\A\leq \alpha^2$. Hence $D_\tau(\A)$ is bounded, and thus compact by the Heine-Borel theorem.
\end{proof}

\begin{theorem}\label{t:quantum-c*-equiv}
  Let $(\A, L)$ be a compact quantum metric space and let $\tau$ be a faithful trace. If $\A$ is   finite dimensional, then $d_L^\tau$ and $d_\A$ are topologically equivalent.
\end{theorem}
\begin{proof}
    Let $(a_n)_{n \in \N}$ be a sequence in $D_\tau(\A)$ that converges to $a\in D_\tau(\A)$ with respect to $d_\A$. Then, for each $b \in \A$, we have that 
\[
\lim_{n\to \infty} \varphi^
\tau_{a_n}(b)=\lim_{n\to \infty} \tau(a_nb)=\tau(ab)=\varphi_a^\tau(b)
\]
since multiplication is continuous and $\tau$ is continuous. 
Thus, $\Phi_\tau$ is continuous and   since $(D_\tau(\A), d_\A)$ is compact by Proposition \ref{p:c*-metric-compact-fd}, we have that $\Phi_\tau$ is a homeomorphism onto its image with respect to the weak* topology on $S(\A)$.

Since convergence in $mk_L$
 is equivalent to weak* convergence by definition, we have that $  d_\A $ and $ d^\tau_L $ are  topologically equivalent. 
\end{proof}

Now, we establish that topological equivalence of the C*-metric and Bures metric.

\begin{theorem}\label{t:bures-c*-equiv}
    Let $\A$ be a unital C*-algebra and let $\tau$ be a faithful trace. If $\A$ is a finite dimensional, then the C*-metric $d_\A$ and the Bures metric $d_B^\tau$ are topologically equivalent.
\end{theorem}
\begin{proof}
    Consider
 \[
 id_{D_\tau(\A)}: (D_\tau(\A), d_\A) \longrightarrow  (D_\tau(\A), d_B^\tau)
 \]
By Theorem \ref{t:conv-bures-c*}, we have that convergence in $d_\A$ implies convergence in $d^\tau_B$. Hence $ id_{D_\tau(\A)}$ is continuous, and since $(D_\tau(\A), d_\A)$ is compact, we have that $ id_{D_\tau(\A)}$ is a homeomorphism. Hence $d_\A$ and $d_B^\tau$ are topologically equivalent.
\end{proof}
We conclude this section with topological equivalence of the Bures metric and quantum metric.

\begin{corollary}\label{c:bures-quantum-equiv}
     Let $(\A,L)$ be a compact quantum metric space and let $\tau$ be a faithful trace. If $\A$ is finite dimensional, then the Bures metric $d_B^\tau$ and quantum metric $d_L^\tau$ are topologically equivalent.
\end{corollary}
\begin{proof}
    This follows immediately from Theorems \ref{t:bures-c*-equiv} and \ref{t:quantum-c*-equiv} and since homeomorphic is an equivalence relation.
\end{proof}

\section{Metric inequivalence for $\C^2$}\label{s:non-metric}

In this last section, we see that although the Bures metric and quantum metric are topologically  equivalent for finite-dimensional C*-algebras, there exists finite-dimensional C*-algebras for which they aren't equivalent as metric spaces. Our approach also proves that the Bures metric and C*-norm are not metrically equivalent since we find a quantum metric that in fact agrees with the metric on the density space induced by the C*-norm.

Consider the unital C*-algebra $\C^2$. Define for every $x=(x_1,x_2)\in \C^2$
\[
L^B(x)=|x_1-x_2|.
\]
It holds that $L^B$ is an L-seminorm since $\C^2$ is finite dimensional and by \cite[Proposition 1.6]{Rieffel98a}. And so $mk_{L^B}$ is a quantum metric 

Consider the following faithful trace $\tau$ on $\C^2$ defined for every $x \in \C^2$ by 
\[
\tau(x)=x_1+x_2.
\]
Now, the Bures metric is given explicitly by the trace, but as the quantum metric is defined by way of a supremum, we first find a formula to explicitly calculate the quantum metric in this case, which in fact recovers the C*-metric.
\begin{theorem}\label{t:explicit-quantum-metric}
    With the setting as above, we have that for every $x,y \in D_\tau\left(\C^2\right)$
    \[
    d_{L^B}^\tau(x,y)=|x_1-y_1|=\|x-y\|_{\C^2}=d_\A(x,y).
    \]
\end{theorem}
\begin{proof}
    First, let $x  \in D_\tau\left(\C^2\right)$ and let $x'\in \C^2$.  Then
    \[
    \Phi_\tau(x)(x')=\phi^\tau_x(x')=\tau(xx')=x_1x'_1+x_2x_2'.
    \]
    Now, assume that $y \in D_\tau\left(\C^2\right)$. Then as $\tau(x)=1=\tau(y)$, we have that $x_1+x_2=1=y_1+y_2$, and so $x_1-y_1=y_2-x_2$.

    Hence, if $L^B(x')\leq 1$, we have
\begin{align*}
    |\varphi^\tau_x(x')-\varphi^\tau_y(x')|& =|x_1x_1'+x_2x_2'-y_1x_1'-y_2x_2'|\\
    & = |(x_1-y_1)x_1'+(x_2-y_2)x_2'|\\
    &=   |(x_1-y_1)x_1'-(y_2-x_2)x_2'|\\
    & =|(x_1-y_1)x_1'-(x_1-y_1)x_2'|\\
    & = |(x_1-y_1)(x_1'-x_2')|\\
    & =|x_1-y_1|\cdot |x_1'-x_2'|\\
    & \leq |x_1-y_1|.
\end{align*}
Now, consider $x'=(1,0)$. Then
\[
|\varphi^\tau_x(x')-\varphi^\tau_y(x')|=|x_1-y_1|.
\]
Thus
\[
d_{L^B}^\tau(x,y)=\sup \{ |\varphi^\tau_x(x')-\varphi^\tau_y(x')| : L^B(x')\leq 1\}=|x_1-y_1|.
\]
We also note that 
\[
\|x-y\|_{\C^2}=\max\{ |x_1-y_1|, |x_2-y_2|\}=\max\{|x_1-y_1|, |x_1-y_1|\}=|x_1-y_1|
\]
 as desired.
\end{proof}

\begin{theorem}\label{t:metric-inequiv}
   Let $d=d_{L^B}^\tau$, the Bures metric, or $d=d_\A$. With the setting as above, it holds that the quantum metric $d_{L^B}^\tau$ and $d$ are topologically equivalent but not metric equivalent.
\end{theorem}
\begin{proof}
    Corollary \ref{c:bures-quantum-equiv} provides topological equivalence. 

    Consider $x=(1,0)\in D_\tau\left(\C^2\right)$. Let $y=(y_1,y_2)\in D_\tau\left(\C^2\right)$.  Then, by Theorem \ref{t:explicit-quantum-metric} $$d_{L^B}^\tau(x,y)=|1-y_1|.$$
    Also, 
  $$   d_B^{\tau}(x,y)=\sqrt{1-\tau(|\sqrt{x}\sqrt{y}|)}=\sqrt{1-\sqrt{y_1}}.$$
Now, let us consider the ratio \[\frac{d_B^\tau(x,y)}{d_{L^B}^\tau(x,y)}=\frac{\sqrt{1-\sqrt{y_1}}}{|1-y_1|}.\] 
But $\lim_{y_1\to 1^-}\sqrt{1-\sqrt{y_1}}=0$ and $\lim_{y_1\to 1^-}1-y_1=0$. 

Now $\frac{d}{dy_1}(\sqrt{1-\sqrt{y_1}})=\frac{-1}{4\sqrt{y_1}\sqrt{1-\sqrt{y_1}}}$ and $\frac{d}{dy_1}(1-y_1)=-1$, and
\[ \lim_{y_1\to1^-} \frac{\frac{-1}{4\sqrt{y_1}\sqrt{1-\sqrt{y_1}}}}{-1}=\lim_{y_1\to1^-}\frac{1}{4\sqrt{y_1}\sqrt{1-\sqrt{y_1}}}=\infty. \]
Thus by L'Hopital's rule, we have 
\[
\lim_{y_1\to 1^-}\frac{d_B^\tau(x,y)}{d_{L^B}^\tau(x,y)}=\infty.
\]
Hence, the set 
\[
\left\{ \frac{d_B^\tau(x,y)}{d_{L^B}^\tau(x,y)}: x,y \in D_\tau\left(\C^2\right), x \neq y\right\}
\]
is unbounded, and so the metrics $d_B^\tau $ and $ d_{L^B}^\tau$ are not metric equivalent.
\end{proof}

\bibliographystyle{amsplain}
\bibliography{thesis}

\providecommand{\bysame}{\leavevmode\hbox to3em{\hrulefill}\thinspace}
\providecommand{\MR}{\relax\ifhmode\unskip\space\fi MR }
\providecommand{\MRhref}[2]{%
  \href{http://www.ams.org/mathscinet-getitem?mr=#1}{#2}
}
\providecommand{\href}[2]{#2}
\begin{thebibliography}{10}

\bibitem{Aguilar24}
Konrad Aguilar and Jiahui Yu, \emph{The {F}ell topology and the modular
  {G}romov-{H}ausdorff propinquity}, Proc. Amer. Math. Soc. \textbf{152}
  (2024), no.~4, 1711--1724. \MR{4709237}

\bibitem{Bures}
Donald Bures, \emph{An extension of {K}akutani's theorem on infinite product
  measures to the tensor product of semifinite {$w\sp{\ast} $}-algebras},
  Trans. Amer. Math. Soc. \textbf{135} (1969), 199--212. \MR{236719}

\bibitem{Davidson}
Kenneth~R. Davidson, \emph{{$C^*$}-algebras by example}, Fields Institute
  Monographs, vol.~6, American Mathematical Society, Providence, RI, 1996.
  \MR{1402012}

\bibitem{Farenick-Rahaman17}
Douglas Farenick and Mizanur Rahaman, \emph{Bures contractive channels on
  operator algebras}, New York J. Math. \textbf{23} (2017), 1369--1393.
  \MR{3723514}

\bibitem{Latremoliere24b}
Carla Farsi, Fr\'{e}d\'{e}ric Latr\'{e}moli\`ere, and Judith Packer,
  \emph{Convergence of inductive sequences of spectral triples for the spectral
  propinquity}, Adv. Math. \textbf{437} (2024), Paper No. 109442, 59.
  \MR{4674861}

\bibitem{Hayashi}
Masahito Hayashi, \emph{Quantum information theory}, second ed., Graduate Texts
  in Physics, Springer-Verlag, Berlin, 2017, Mathematical foundation.
  \MR{3558531}

\bibitem{Kaad20}
Jens Kaad and David Kyed, \emph{Dynamics of compact quantum metric spaces},
  Ergodic Theory Dynam. Systems \textbf{41} (2021), no.~7, 2069--2109,
  arXiv:1904.13278. \MR{4266364}

\bibitem{Latremoliere-Packer17}
{F} {L}atr\'emoli\`ere and {J} {P}acker, \emph{Noncommutative solenoids and the
  {G}romov-{H}ausdorff propinquity}, Proc. Amer. Math. Soc. \textbf{145}
  (2017), no.~5, 2043--2057. \MR{3611319}

\bibitem{Latremoliere20}
Fr\'{e}d\'{e}ric Latr\'{e}moli\`ere, \emph{The covariant {G}romov-{H}ausdorff
  propinquity}, Studia Math. \textbf{251} (2020), no.~2, 135--169, arXiv:
  1805.11229. \MR{4045657}

\bibitem{Rieffel98a}
M.~A. {R}ieffel, \emph{Metrics on states from actions of compact groups}, Doc.
  Math. \textbf{3} (1998), 215--229. \MR{1647515}

\bibitem{Rieffel00}
\bysame, \emph{Gromov-{H}ausdorff distance for quantum metric spaces}, vol.
  168, 2004, Appendix 1 by Hanfeng Li, Gromov-Hausdorff distance for quantum
  metric spaces. Matrix algebras converge to the sphere for quantum
  Gromov-Hausdorff distance, pp.~1--65. \MR{2055927}

\bibitem{Rieffel01}
\bysame, \emph{Matrix algebras converge to the sphere for quantum
  {G}romov--{H}ausdorff distance}, Mem. Amer. Math. Soc. \textbf{168} (2004),
  no.~796, 67--91, math.OA/0108005.

\end{thebibliography}

 \vfill
\end{document}